\documentclass{article}

\usepackage[letterpaper,margin=1in]{geometry}
\usepackage{mathtools, verbatim, amsmath, amsfonts, amsthm, amssymb, hyperref, xcolor}
\usepackage[shortlabels]{enumitem}
\usepackage[matrix,arrow]{xy}

\numberwithin{equation}{section}

\newtheorem{thm}{Theorem}[section]
\newtheorem{lem}[thm]{Lemma}

\newtheorem{prop}[thm]{Proposition}

\theoremstyle{definition}

\newtheorem{defn}[thm]{Definition}
\newtheorem{qn}[thm]{Question}
\newtheorem{examp}[thm]{Example}
\theoremstyle{remark}
\newtheorem{rem}[thm]{Remark}

\DeclareMathOperator{\tr}{tr}
\DeclareMathOperator{\id}{id}

\DeclareMathOperator{\sgn}{sgn}
\DeclareMathOperator{\ch}{char}

\DeclareMathOperator{\disc}{disc}

\DeclareMathOperator{\inv}{inv}

\newcommand{\ur}{\mathrm{ur}}

\newcommand{\sep}{{\mathrm{sep}}}
\DeclareMathOperator{\Spec}{Spec}
\DeclareMathOperator{\Res}{Res}
\DeclareMathOperator{\Cor}{Cor}
\DeclareMathOperator{\Hom}{Hom}

\DeclareMathOperator{\Coord}{Coord}
\DeclareMathOperator{\Kum}{Kum}
\DeclareMathOperator{\Ind}{Ind}

\DeclareMathOperator{\Hol}{Hol}
\DeclareMathOperator{\Sym}{Sym}
\DeclareMathOperator{\Aut}{Aut}
\DeclareMathOperator{\Gal}{Gal}

\newcommand{\CC}{\mathbb{C}}
\newcommand{\FF}{\mathbb{F}}
\newcommand{\QQ}{\mathbb{Q}}
\newcommand{\ZZ}{\mathbb{Z}}

\newcommand{\A}{\mathcal{A}}
\newcommand{\F}{\mathcal{F}}

\renewcommand{\L}{{\mathrm{L}}}
\newcommand{\R}{{\mathrm{R}}}

\newcommand{\cross}{\times}
\newcommand{\tensor}{\otimes}
\newcommand{\textand}{\quad \text{and} \quad}

\renewcommand{\to}{\mathop{\rightarrow}\limits}
\newcommand{\longto}{\mathop{\longrightarrow}\limits}
\newcommand{\size}[1]{\lvert #1 \rvert}

\newcommand{\isom}{\cong}
\newcommand{\<}{\left\langle}
\renewcommand{\>}{\right\rangle}
\renewcommand{\(}{\left(}
\renewcommand{\)}{\right)}
\newcommand{\laurent}[1]{(\!(#1)\!)}

\newcommand{\ignore}[1]{}

\renewcommand{\epsilon}{\varepsilon}

\title{\'Etale algebras and the Kummer theory of finite Galois modules}
\author{Evan M. O'Dorney}

\begin{document}
  
\maketitle

\begin{abstract}
  Galois cohomology groups $H^i(K,M)$ are widely used in algebraic number theory, in such contexts as Selmer groups of elliptic curves, Brauer groups of fields, class field theory, and Iwasawa theory. The standard construction of these groups involves maps out of the absolute Galois group $G_K$, which in many cases of interest (e.g.\ $K = \mathbb{Q}$) is too large for computation or even for gaining an intuitive grasp. However, for finite $M$, an element of $H^i(K,M)$ can be described by a finite amount of data. For the important case $i = 1$, the appropriate object is an \'etale algebra over $K$ (a finite product of separable field extensions) whose Galois group is a subgroup of the semidirect product $\operatorname{Hol} M = M \rtimes \operatorname{Aut} M$ (often called the \emph{holomorph} of $M$), equipped with a little bit of combinatorial data.
  
  Although the correspondence between $H^1$ and field extensions is in widespread use, it includes some combinatorial and Galois-theoretic details that seem never to have been written down. In this short quasi-expository paper, we fill in this gap and explain how the \'etale algebra perspective illuminates some common uses of $H^1$, including parametrizing cubic and quartic algebras as well as computing the Tate pairing on Galois coclasses of local fields.
\end{abstract}

\section{Introduction}

Galois cohomology groups $H^i(K,M)$ are widely used in algebraic number theory, in such contexts as Selmer groups of elliptic curves, Brauer groups of fields, class field theory, and Iwasawa theory. Many references exist for the subject (\cite{Berhuy_2010,SerreGC}; see also \cite[Part 3]{SerreLF}). The standard construction of these groups involves maps out of the absolute Galois group $G_K$, which in many cases of interest (e.g.\ $K = \QQ$) is too large for computation or even for gaining an intuitive grasp. However, for finite Galois modules $M$, an element of $H^i(K,M)$ can be described by a finite amount of data over $K$. For the important case $i = 1$, the appropriate object is an \'etale algebra over $K$ (a finite product of finite separable field extensions) whose Galois group is a subgroup of the semidirect product $\Hol M = M \rtimes \Aut M$ (often called the \emph{holomorph} of $M$), equipped with a little bit of combinatorial data.

Although the correspondence between $H^1$ and field extensions is in widespread use, it includes some combinatorial and Galois-theoretic details that seem never to have been written down. In this short paper, we fill in this gap and explain how the \'etale algebra perspective illuminates some common uses of $H^1$.

In section \ref{sec:cochains}, we review the usual construction of Galois cohomology via cochains. In section \ref{sec:etale}, we summarize the theory of \'etale algebras, which is covered in greater detail by Milne (\cite{MilneFields}, chapter 8). In section \ref{sec:gcoho}, we explain how \'etale algebras yield explicit descriptions of $H^0(K, M)$ and $H^1(K,M)$. In section \ref{sec:Kummer}, we restrict to the cases $\size{M} \leq 4$ and describe $H^1(K,M)$ even more explicitly via Kummer theory.
Lastly, in section \ref{sec:tate}, we explain how the local Tate pairing, defined formally as a cup product
\[
  \<\bullet,\bullet\> : H^1(K, M) \cross H^1(K, M') \to H^2(K, \mu_m) \isom C_m,
\]
devolves in some of these cases into the local Hilbert symbol, which describes the solubility of norm equations (e.g.\ conics) over a local field.

\subsection{Notation}

If $X$ is a finite set, we denote by $\Sym X$ the group of bijections from $X$ to itself (isomorphic to $S_{\size{X}}$).

We use the semicolon to separate coordinates in a direct product of rings, so for instance, the nontrivial idempotents in $\ZZ \cross \ZZ$ are $(1;0)$ and $(0;1)$.

\section{Galois cohomology: the classical definition}
\label{sec:cochains}

\subsection{Galois groups and Galois modules}
Let $K$ be a field, let $K^\sep$ be the separable closure of $K$, and denote by $G_K = \Gal(K^\sep/K)$ the absolute Galois group. Note that $G_K$ is profinite when equipped with the observable topology, that is, basic opens are left cosets $U = g G_L$ for finite separable extensions $L$ of $K$. Equivalently, a basic open $U$ can be specified by choosing finitely many elements $x_1,\ldots,x_n,y_1,\ldots,y_n \in K^\sep$ and letting
\[
  U = \{g \in G_K : g(x_i) = y_i\}.
\]

A \emph{Galois module} over $K$ is an abelian Hausdorff topological group $M$ equipped with a continuous action of $G_K$, that is, a continuous map $\phi : G_K \to \Aut M$. Galois cohomology is often defined only for \emph{discrete} modules to streamline the proofs \cite{SerreGC,NSW}. The focus of this paper will be on \emph{finite} Galois modules, but see Section \ref{sec:profinite} below.

As group cohomology is a purely algebraic notion, one might be inclined to regard the topological considerations as an added annoyance; but for $G_K$, the topology actually makes things simpler:
\begin{prop}
If $M$ is a finite abelian group, then $K$-Galois module structures on $M$ are in natural bijection with pairs consisting of
\begin{enumerate}[$($a$)$]
  \item a finite Galois extension $L/K$, and
  \item an embedding $\Gal(L/K) \hookrightarrow \Aut M$.
\end{enumerate}
\end{prop}
\begin{proof}
If $M$ is a finite $K$-Galois module, then by continuity, the kernel of the structure map $\phi : G_K \to \Aut M$ is a closed subgroup of $G_K$, necessarily of finite index. In particular, $\phi$ factors through $\Gal(E/K)$ for some (finite) Galois extension $E/K$. Let $L$ be the fixed field of the kernel of $\tilde{\phi} : \Gal(E/K) \to \Aut M$, which is necessarily also Galois over $K$, thereby obtaining the desired map. Observe that the allowable fields $E$ are those that are Galois over $K$ and contain $L$, and that $L$ is independent of the choice of $E$, being the fixed field of $\ker \phi$.

Conversely, an embedding $\Gal(L/K) \hookrightarrow \Aut M$ defines a Galois module structure on $M$ by composition with the restriction map $G_K \to \Gal(L/K)$. It is evident that these two constructions are inverse.
\end{proof}

\begin{examp} \label{ex:cubic_Gal_mod}
  Let $M = C_3$ be the smallest group with nontrivial automorphism group: $\Aut M \isom C_2$. Then the nontrivial Galois module structures on $M$ are in natural bijection with separable quadratic extensions $T/K$, as every quadratic \'etale extension $T$ necessarily has a unique involution $x \mapsto \tr_{T/K} x - x$. If $\ch K \neq 2$, these $T$ can be parametrized by Kummer theory as $T = K[\sqrt{D}]$, $D \in K^\cross/\( K^\cross\)^2$. The value $D = 1$ corresponds to the split algebra $T = K \cross K$ and to the module $M$ with trivial action. We have an isomorphism
  \[
  M_T \cong \{0, \sqrt{D}, -\sqrt{D}\}
  \]
  of sets with $G_K$-action, and of Galois modules if the right-hand side is given the appropriate group structure with $0$ as identity.
  
  In particular, the Galois-module structures on $C_3$ form a group $\Hom(G_K, C_2) \isom K^\cross/\(K^\cross\)^2$: the group operation can also be viewed as \emph{tensor product} of one-dimensional $\FF_3$-vector spaces with Galois action.
\end{examp}

\subsection{Profinite Galois modules} \label{sec:profinite}
Another situation of interest is when $M = \varprojlim_{i \in I} M_i$ is a \emph{profinite} group. In favorable cases (such as $M = \ZZ_p^n$, which occurs for instance as the Tate module of an elliptic curve), the maps $\pi_{ji} : M_j \to M_i$ are surjective with kernel a \emph{characteristic} subgroup of $M_j$ (that is, fixed under all automorphisms of $M_j$), thereby inducing a map $\pi_{ji*} : \Aut M_j \to \Aut M_i$. In this case, a Galois module structure on $M$ is a compatible system of continuous maps from $G_K$ to each $\Aut M_i$, which may be given by extensions $L_i/K$ and embeddings $\phi_i : \Gal(L_i/K) \to \Aut M_i$ such that if $j \succ i$, then $L_j \supseteq L_i$ and the diagram
\[
\xymatrix{
  \Gal(L_j/K) \ar[r]^{\phi_j}\ar[d]_{\bullet|_{L_i}} & \Aut M_j\ar[d]^{\pi_{ji*}} \\
  \Gal(L_i/K) \ar[r]_{\phi_i} & \Aut M_i
}
\]
commutes. The cohomology of $M$ can then be computed using continuous cochains. We will not discuss this case further; see \cite[Appendix B]{RubinES} for a fuller treatment.

\subsection{Galois cohomology}
If $M$ is a $K$-Galois module, then for $n \in \ZZ_{\geq 0}$, an \emph{$n$-cochain} of $M$ is a continuous map from $G_K^i$ to $M$. If $M$ is finite, then by continuity, an $n$-cochain factors through $\Gal(L/K)^n$ for some finite Galois extension $L/K$. The $i$-cochains form an abelian group $C^n(K, M)$ under addition of outputs in $M$. We naturally have $C^0(K,M) = M$. We define a \emph{differential,} or \emph{coboundary map,} $\partial = \partial^n : C^n(K, M) \to C^{n+1}(K, M)$ by
\[
  \partial^n(\sigma)\(g_1,\ldots,g_{n+1}\) = g_1 \cdot \sigma(g_2,\ldots,g_{n+1}) +
  \sum_{i=1}^n
  (-1)^i\sigma(g_1,\ldots,g_i g_{i+1},\ldots,g_{n+1})
  +(-1)^{n+1}\sigma(g_1,\ldots,g_n).
\]
For instance,
\begin{align*}
  \partial^0(u)(g) &= g \cdot u - u \\
  \partial^1(\sigma)(g,h) &= g \cdot \sigma(h) - \sigma(gh) + \sigma(g) \\
  \partial^2(\sigma)(g,h,k) &= g \cdot \sigma(h,k) - \sigma(g h, k) + \sigma(g, h k) - \sigma(g, h).
\end{align*}
It is a straightforward computation to show that $\partial \circ \partial = 0$, yielding a chain complex
\[
  0 \longto C^0(K,M) \longto^{\partial^0} C^1(K,M) \longto^{\partial^{-1}} C^2(K,M) \longto \cdots\cdot
\]
We utilize the usual terminology for chain complexes:
\begin{itemize}
  \item The kernel of $\partial^{i}$ is called $Z^i(K,M)$, the group of \emph{$i$-cocycles.}
  \item The image of $\partial^{i-1}$ is called $B^i(K,M)$, the group of \emph{$i$-coboundaries.} We set $B^0(K,M) = 0$.
  \item The quotient $Z^i(K,M)/B^i(K,M)$ is called $H^i(K,M)$, the $i$th \emph{Galois cohomology group} of $M$. Elements of $H^i(K,M)$ are called \emph{cohomology classes} or simply \emph{coclasses.}\footnote{Credit goes to Brandon Alberts for coining this handy contraction.}
\end{itemize}
By definition, a coclass is represented by a cocycle, which is a continuous map from $G_K^i$ to $M$. If $M$ is finite, then the map factors through a finite quotient $\Gal(L/K)^i$. But the field $L$ may grow larger or smaller as the representative cocycle is translated by coboundaries, and there may not be a unique smallest field of definition $L$ for a coclass if $i \geq 2$.

For $i = 0$, we have that $H^0$ is the fixed-points functor:
\[
  H^0(K, M) = M^{G_K} = \{u \in M : g u = u \forall g \in G_K\}.
\]
As is typical of cohomology theories, one of the main benefits of Galois cohomology is that any short exact sequence of $K$-Galois modules
\[
  0 \to L \to M \to N \to 0
\]
yields a long exact sequence
\[
  0 \to H^0(K,L) \to H^0(K,M) \to H^0(K,N)
    \to H^1(K,L) \to H^1(K,M) \to H^1(K,N)
    \to \cdots\cdot
\]

\section{\'Etale algebras and their Galois groups}
\label{sec:etale}

If $K$ is a field, an \emph{\'etale algebra} over $K$ is a finite-dimensional separable commutative algebra over $K$, or equivalently, a finite product of finite separable extension fields of $K$. Typical examples are given by $L = K[x]/(f(x))$, where $f \in K[x]$ is a separable polynomial. A thorough treatment of \'etale algebras is found in Milne (\cite{MilneFields}, chapter 8): here we summarize this theory and prove a few auxiliary results that will be of use.

Fix a ground field $K$, its separable closure $K^\sep$, and let $G_K = \Gal(K^\sep/K)$ be the absolute Galois group. An \'etale algebra $L$ of rank $n$ admits exactly $n$ maps $\iota_1,\ldots, \iota_n$ (of $K$-algebras) to $K^\sep$. We call these the \emph{coordinates} of $L$; the set of them will be called $\Coord_K(L)$ or simply $\Coord(L)$. Together, the coordinates define an embedding of $L$ into $\(K^\sep\)^n$, which we call the \emph{Minkowski embedding} because it subsumes as a special case the embedding of a degree-$n$ number field into $\CC^n$, which plays a major role in algebraic number theory, as in Delone--Faddeev \cite{DF}.

For any element $g$ of the absolute Galois group $G_K$, the composition $g \circ \iota_i$ with any coordinate is also a coordinate $\iota_j$, so we get a homomorphism $\phi = \phi_L : G_K \to \Sym(\Coord(L)) \isom S_n$ defined by
\[
(\phi_g\iota)(x) = g(\iota(x))
\]
for all $x \in L, \iota \in \Coord(L)$. This gives a functor from \'etale $K$-algebras to $G_K$-sets (sets with a $G_K$-action), which is denoted $\F$ in Milne's terminology. A functor going the other way, which Milne calls $\A$, takes $\phi : G_K \to S_n$ to
\begin{equation} \label{eq:Gset_to_etale}
  L = \{ (x_1,\ldots,x_n) \in \(K^\sep\)^n \mid x_{\phi_g(i)} = g(x_i) \, \forall g \in G_K, \forall i \}
\end{equation}
Although it is not immediately obvious, this $L$ is an $n$-dimensional $K$-algebra, Minkowski-embedded into $\(K^\sep\)^n$.

\begin{prop}[\cite{MilneFields}, Theorem 7.29] \label{prop:G_sets}
  The functors $\F$ and $\A$ establish a bijection between
  \begin{itemize}
    \item \'etale extensions $L/K$ of degree $n$, up to isomorphism, and
    \item $G_K$-sets of size $n$ up to isomorphism; that is to say, continuous homomorphisms $\phi : G_K \to S_n$, up to conjugation in $S_n$.
  \end{itemize}
\end{prop}

In this bijection, the \emph{transitive} $G_K$-sets correspond to \'etale algebras that are fields. Moreover, the bijection respects base change, in the following way:

\begin{prop}\label{prop:respects_base_chg}
  Let $K_1/K$ be a field extension, not necessarily algebraic, and let $L/K$ be an \'etale extension of degree $n$. Then $L_1 = L \tensor_{K} K_1$ is \'etale over $K_1$, and the associated Galois permutation representations $\phi_{L/K}$, $\phi_{L_1/K_1}$ are related by the commutative diagram
  \begin{equation}\label{eq:respects_base_chg}
    \xymatrix{
      G_{K_1} \ar[r]^{\bullet |_{K^\sep}} \ar[d]_{\phi_{L_1/K_1}} & G_K \ar[d]^{\phi_{L/K}} \\
      \Sym(\Coord_{K_1}(L_1)) \ar@{-}[r]^{\sim} & \Sym(\Coord_K(L))
    }
  \end{equation}
\end{prop}
\begin{proof}
That $L_1/K_1$ is \'etale is standard (see Milne \cite{MilneFields}, Prop.~8.10). For the second claim, consider the natural restriction map $r : \Coord_{K_1}(L_1) \to \Coord_K(L)$. It is injective, since a $K_1$-linear map out of $L_1$ is determined by its values on $L$; and since both sets have the same size, $r$ is surjective and is hence an isomorphism of $G_{K_1}$-sets (the $G_{K_1}$-structure on $\Coord_{K}(L)$ arising by restriction from the $ G_{K} $-structure).
\end{proof}

\subsection{The Galois group of an \'etale algebra}

Define the \emph{Galois group} $G(L/K)$ of an \'etale algebra to be the image of its associated Galois representation $\phi : G_K \to \Sym(\Coord(L)) \isom S_n$. It transitively permutes the coordinates corresponding to each field factor. If $L$ is a field, then $G(L/K)$ is the Galois group of its normal closure, equipped with an embedding into $S_n$. For example, if $L$ is a quartic field, then $G(L/K)$ is one of the five (up to conjugacy) transitive subgroups of $S_4$, which (to use the traditional names) are $S_4$, $A_4$, $D_4$, $V_4$, and $C_4$. Galois groups in this sense are used in the tables of cubic and quartic fields in Delone-Faddeev \cite{DF} and the Number Field Database \cite{NFDB}. Note that the Galois group $G(L/K)$ is defined whether or not $L$ is a Galois extension. If it is, then the Galois group is \emph{simply} transitive and coincides with the Galois group in the sense of Galois theory.

The following notion will be useful for us.
\begin{defn}
  Let $G \subseteq S_n$ be a subgroup. A \emph{$G$-extension} of $K$ is a degree-$n$ \'etale algebra $L$ with a choice of subgroup $G' \subseteq \Sym(\Coord(L))$ that is conjugate to $G$ (in $\Sym(\Coord(L))$) and contains $G(L/K)$, plus a $G$-conjugacy class of isomorphisms $G' \cong G$. The added data is called a \emph{$G$-structure} on $L$.
\end{defn}
\begin{prop} \label{prop:G_sets_G_strucs}
  Let $G \subseteq S_n$ be a subgroup. $G$-extensions $L/K$ up to isomorphism are in bijection with homomorphisms $\phi : G_K \to G$, up to conjugation in $G$.
\end{prop}
\begin{proof}
  By Proposition \ref{prop:G_sets}, an \'etale algebra $L$ corresponds to a map $\phi : G_K \to S_n$. The added data of a $G$-extension furnishes an group $G' \hookrightarrow S_n$ containing the image of $\phi$, with a $G$-conjugacy class of isomorphisms $\theta : G' \cong G$. This gives a map $\theta \circ \phi : G_K \to G$, unique up to conjugation in $G$. Conversely, given a map $\phi : G_K \to G$, we compose with the natural embedding $G \subseteq S_n$ to get an \'etale algebra $L$ which we equip with the subgroup $G$ and the identity map $\id : G \isom G$. The two constructions are inverse up to relabeling the coordinates of $L$, that is, conjugation in $S_n$.
\end{proof}

\begin{examp}\label{ex:C4-structure}
$L = \QQ(\zeta_5)$ is a $C_4$-extension (taking $C_4 = \<(1234)\> \subseteq S_4$), indeed its Galois group is isomorphic to $C_4$; and $L$ admits two distinct $C_4$-structures, as there are two ways to identify $C_4$ with its image in $S_4$, which are conjugate in $S_4$ but not in $C_4$. Likewise, $L = \QQ \cross \QQ \cross \QQ \cross \QQ$ admits six $C_4$-structures, one for each embedding of $C_4$ into $S_4$, as $G(L/\QQ)$ is trivial.
\end{examp}

\subsection{Resolvents}
The solution of a general quartic equation by radicals inevitably passes through a \emph{resolvent cubic} which must be solved first. Similarly, whether a quintic equation is solvable can be told from its sextic resolvent. The \'etale algebras generated by the roots of these resolvents are instances of the following definition.

\begin{defn}
Let $G \subseteq S_n$, $H \subseteq S_m$ be subgroups and $\rho : G \to H$ be a homomorphism. Then for every $G$-extension $L/K$, the corresponding $\phi_L : G_K \to G$ may be composed with $\rho$ to yield a map $\phi_R : G_K \to H$, which defines an $H$-extension $R/K$ of degree $m$. This $R$ is called the \emph{resolvent} of $L$ under the map $\rho$.
\end{defn}

\begin{examp}\label{ex:rsv43}
Since there is a surjective map $\rho_{4,3} : S_4 \to S_3$, every quartic \'etale algebra $L/K$ has a cubic resolvent $R$, generated by a formal root of the \emph{resolvent cubic} that appears when a general quartic equation is to be solved by radicals. There is also an analogue of this construction over $\ZZ$; see Bhargava \cite{B3}.
\end{examp}

\begin{examp} \label{ex:rsv_n2}
Likewise, the sign map can be viewed as a homomorphism $\sgn : S_n \to S_2$, attaching to every \'etale algebra $L$ a quadratic resolvent $T$. If $L = K[\theta]/f(\theta)$ is generated by a polynomial $f$, and if $\ch K \neq 2$, then it is not hard to see that $T = K[\sqrt{\disc f}]$ where $\disc f$ is the polynomial discriminant. Note that $T$ still exists even if $\ch K = 2$. We have that $T \cong K \cross K$ is split if and only if the Galois group $G(L/K)$ is contained in the alternating group $A_n$.
\end{examp}

\begin{examp}\label{ex:D4}
The dihedral group $D_4$ of order $8$ has an outer automorphism, because rotating a square in the plane by $45^\circ$ does not preserve the square but does preserve every symmetry of the square. This map $\rho : D_4 \to D_4$ associates to each $D_4$-algebra $L$ a new $D_4$-algebra $L'$, not in general isomorphic. This is the classical phenomenon of the \emph{mirror field}. For instance, if $L = \QQ[\sqrt{3 + \sqrt{2}}]$, then
\[
  L' = \QQ\left[\sqrt{3 + \sqrt{2}} + \sqrt{3 - \sqrt{2}}\right] = \QQ\left[\sqrt{6 + 2\sqrt{7}}\right].
\]
Both $L$ and $L'$ have the same Galois closure, the $D_4$-octic extension $\tilde{L} = \QQ\left[\sqrt{3 + \sqrt{2}}, \sqrt{7}\right]$.
\end{examp}

\begin{examp}
Likewise, the outer automorphism $\phi$ of $S_6$ permits the association to each sextic \'etale algebra $L/K$ a mirror sextic \'etale algebra $L'$. Because $\phi$ sends any transposition to a product of three disjoint transpositions, we have, for $Q/K$ quadratic, the mirror algebra association
\[
  Q \cross K \cross K \cross K \cross K \quad
  \longleftrightarrow
  \quad Q \cross Q \cross Q.
\]
\end{examp}

\begin{examp} \label{ex:Cayley}
The \emph{Cayley embedding} is an embedding of any group $G$ into $\Sym G$, acting by left multiplication. The Cayley embedding $\rho : S_n \hookrightarrow S_{n!}$ attaches to every \'etale algebra $L$ of degree $n$ an algebra $\tilde L$ of degree $n!$ (with an $S_n$-action, as we shall soon see). This is none other than the \emph{$S_n$-closure} of $L$, constructed by Bhargava in a quite different way in \cite[Section 2]{B3}.

More generally, for any $G \subseteq S_n$, the Cayley embedding $G \hookrightarrow \Sym G$ allows one to associate to each $G$-extension $L$ a degree-$|G|$ extension $T$, which we may call the \emph{$G$-closure} of $L$. The name ``closure'' is justified by the following observation: if $G \subseteq S_n$ is a transitive subgroup, then, since any transitive $G$-set is a quotient of the simply transitive one, we can embed $L$ into $T$ by Proposition \ref{prop:sub} below. More generally, $G$-closures of ring extensions, not necessarily \'etale or even reduced, have been constructed and studied by Biesel \cite{Biesel_thesis, Biesel}.
\end{examp}

If $\rho : G \to H$ is invertible, as in many of the above examples, then the map from $G$-extensions to $H$-extensions is also invertible: we say that the two extensions are \emph{mutual resolvents}.

\subsection{Subextensions and automorphisms}

The Galois group holds the answers to various natural questions about an \'etale algebra. The next two propositions are given without proof, since they follow immediately from the functorial character of the correspondence in Proposition \ref{prop:G_sets}.
\begin{prop} \label{prop:sub}
The subextensions $L' \subseteq L$ of an \'etale extension $L/K$ correspond to the equivalence relations $\sim$ on $\Coord(L)$ stable under permutation by $G(L/K)$, under the bijection
\[
  \mathord{\sim} \mapsto L' = \{ x \in L : \iota(x) = \iota'(x) \text{ whenever } \iota \sim \iota' \}.
\]
\end{prop}
\begin{examp}\label{ex:wreath}
  If an \'etale extension $M/K$ of degree $tb$ has a subextension $L/K$ of degree $b$, with $M$ of constant degree $t$ over each field factor of $L$, then $G(L/K)$ must be contained in the wreath product $S_t \wr S_b$, which is the group of permutations of a $tb$-element set preserving a partition into $b$ subsets of order $t$. This fact is well known for field extensions. 
\end{examp}
\begin{rem}
Note that if $L$ is a Galois field extension, the image of $\phi_L$ is a simply transitive subgroup $\Gamma$, and identifying $\Coord(L)$ with $\Gamma$, the stable equivalence relations are just right congruences modulo subgroups of $\Gamma$: so we recover the Galois correspondence between subgroups and subfields.
\end{rem}

The Galois group $G(L/K)$ is \emph{not} in general a group of automorphisms of $L$. However, the automorphisms of $L$ as a $K$-algebra can be described in terms of the Galois group readily.

\begin{prop}\label{prop:aut}
Let $L$ be Minkowski-embedded by its coordinates $\iota_1,\ldots,\iota_n$. Then the automorphism group $\Aut(L/K)$ is given by permutations of coordinates,
\[
  \tau_\pi(x_1;\ldots;x_n) = (x_{\pi^{-1}(1)} ; \ldots ; x_{\pi^{-1}(n)}),
\]
for $\pi$ in the centralizer $C(S_n, G(L/K))$ of the Galois group.
\end{prop}

This provides a characterization, in terms of the Galois group, of rings having various kinds of automorphisms.

\begin{examp}
  Since $S_2$ is abelian, any \'etale algebra $L$ of rank $2$ has a unique non-identity automorphism, the conjugation $\bar x = \tr_{L/K} x - x$.
\end{examp} 
\begin{examp}
  If $L$ has rank $4$, automorphisms $\tau$ of $L$ of order $2$ whose fixed algebra is of rank $2$ are in bijection with $D_4$-structures on $L$. Indeed, the corresponding permutation $\pi$ must have just two orbits and is therefore conjugate to $(13)(24)$, and the centralizer of this permutation is $D_4$.
\end{examp}

\subsection{Torsors}

If an \'etale algebra is a generalization of a field extension, it is desirable to have a suitable analogue of a \emph{Galois} extension. Here is a workable definition.

\begin{defn}
Let $G$ be a finite group. A \emph{$G$-torsor} over $K$ is an \'etale algebra $L$ over $K$ equipped with an action of $G$ by automorphisms $\{\tau_g\}_{g \in G}$ that permute the coordinates simply transitively, that is, such that $L \tensor_K K^\sep$ is isomorphic to
\[
  \bigoplus_{g \in G} K^\sep
\]
with $G$ acting by right multiplication on the indices.
\end{defn}

\begin{rem}
The usage of the word ``torsor'' in this context is not quite standard; one says instead that $\Spec L$ is a $G$-torsor as a finite flat scheme, but we will not use the scheme-theoretic viewpoint here.
\end{rem}

\begin{prop}
\label{prop:torsor}
Let $G$ be a group of order $n$. An \'etale algebra $L$ is a $G$-torsor if and only if it is a $G$-extension, where $G$ is embedded into $S_n$ by the Cayley embedding ($G$ acting on itself by left multiplication). Moreover, there is a bijection between
\begin{itemize}
  \item $G$-torsor structures on $L$, up to conjugation in $G$, and
  \item $G$-structures on $L$.
\end{itemize}
The bijection is given in the following way: there is a labeling $\{\iota_g\}$ of the coordinates of $L$ with the elements of $G$ such that the Galois action is by left multiplication
\begin{equation} \label{eq:torsor_left}
  g(\iota_h(x)) = \iota_{g h}(x)
\end{equation}
while the torsor action is by right multiplication
\begin{equation} \label{eq:torsor_right}
  \iota_g(\tau_{h}(x)) = \iota_{g h}(x).
\end{equation}
\end{prop}

\begin{proof}
Denote by $G_L$, $G_R$ the embedded images of $G$ in $\Sym G$ given by left and right multiplication, respectively. It is well known that $G_\L$ and $G_\R$ are centralizers of each other \cite[Theorem 6.3.1]{HallGroups}. For completeness, we include the short proof: if $\pi \in C(G_\L)$, then for all $g \in G$,
\[
  \pi(g) = \pi(g \cdot 1_G) = g \cdot \pi(1_G),
\]
so $\pi$ is right multiplication by $\pi(1_G)$, and conversely. Note that $G_\L$ and $G_\R$ are conjugate under the inversion permutation
\begin{align*}
  \epsilon &\in \Sym G \\
  \epsilon(g) &= g^{-1}.
\end{align*}

By Proposition \ref{prop:aut}, $G$-torsor structures on $L$, up to $G$-conjugacy, are in bijection with conjugates $G'$ of $G_\L$ in $\Sym(\Coord(L))$ that commute with $G(L/K)$. By passing to the centralizer, these are in bijection with conjugates $G''$ of $G_\R$ in $\Sym(\Coord(L))$ that \emph{contain} $G(L/K)$, and these are $G$-structures on $L$ by definition.
\end{proof}

Here is another perspective on torsors.
\begin{prop} \label{prop:fld_fac}
$G$-torsors over a field $K$, up to isomorphism, are determined by their field factor, a Galois extension $L_1/K$ equipped with an embedding $\Gal(L/K) \hookrightarrow G$ up to conjugation in $G$.
\end{prop}
\begin{proof}
If $T$ is a $G$-torsor, then since $G$ permutes the coordinates simply transitively, all the coordinates have the same image; that is, the field factors of $G$ are all isomorphic to a Galois extension $L/K$. Fix one coordinate $\iota_1 : T \to L \subseteq K^\sep$, which necessarily factors through one field factor $L_1$ of $T$. The torsor automorphisms permute the field factors of $T$; the subgroup $H$ of those that send $L_1$ to itself maps isomorphically to the Galois group $\Gal(L/K)$. If we choose instead a different coordinate $\iota_i$, the map $\Gal(L/K) \to G$ and its image $H$ get conjugated by the unique $g \in G$ for which $g \circ \iota_1 = \iota_i$.

Conversely, suppose $L$ and an embedding
\[
  \Gal(L/K) \longto^\sim H \subseteq G
\]
are given. Let $1 = g_1,\ldots, g_r$ be coset representatives for $G/H$. Then $g_2,\ldots, g_r$ must map any field factor $L_1 \cong L$ isomorphically onto the remaining field factors $L_2,\ldots,L_r$, each $L_i$ occurring once. To finish specifying the $G$-action on $T \cong L_1 \cross \cdots \cross L_r$, it suffices to determine $g|_{L_i}$ for each $g \in G$. Factor $g g_i = g_j h$ for some $j \in \{1,\ldots,r\}$, $h \in H$. Then for each $x \in L_1$, $g(g_i(x)) = g_j(h(x))$, and the value of this is known because the $H$-action on $L_1$ is known. Then $g|_T$ is an automorphism. To show that this gives a torsor structure, we must show that for $g, g' \in G$, we have $(gg')|_T = g|_T \circ g'|_T$. We restrict our attention to one field factor $L_i$. Let $h, h' \in H$ and $j, k \in \{1,\ldots,n\}$ such that
\[
  g' g_i = g_j h' \textand g g_j = g_k h.
\]
Then $g g' g_i = g g_j h' = g_k h h'$, and $(g g')|_{L_i} = h h' = g \circ g' |_{L_i}$ as maps from $L_i$ to $L_k$.
\end{proof}

\section{Galois cohomology: \texorpdfstring{$H^0$ and $H^1$}{H⁰ and H¹}}
\label{sec:gcoho}
In this section, we explain how to understand the first two Galois cohomology groups of a Galois module. We begin by describing Galois modules.

\begin{prop}[\textbf{a description of Galois modules}] \label{prop:Gal_mod}
  Let $M$ be a finite abelian group, and let $K$ be a field. Let $M^-$ denote the subset of elements of $M$ of maximal order $m$, the exponent of $M$. The following objects are in bijection:
  \begin{enumerate}[$($a$)$]
    \item \label{it:Gal_mod} Galois module structures on $M$ over $K$, that is, continuous homomorphisms $\phi : G_K \to \Aut M$;
    \item \label{it:Gal_mod_T} $(\Aut M)$-torsors $T/K$;
    \item \label{it:Gal_mod_L0} $(\Aut M)$-extensions $L_0/K$ of degree $|M|$, where $\Aut M \hookrightarrow \Sym M$ in the natural way;
    \item \label{it:Gal_mod_L-} $(\Aut M)$-extensions $L^-/K$ of degree $|M^-|$, where $\Aut M \hookrightarrow \Sym M^-$ in the natural way.
  \end{enumerate}
\end{prop}
\begin{proof}
  In writing the map $\Aut M \rightarrow \Sym M^-$ in item \ref{it:Gal_mod_L-} as an injection, we need that $M^-$ generates $M$; this follows easily from the classification of finite abelian groups.
  
The bijections are immediate from Propositions \ref{prop:G_sets_G_strucs} and \ref{prop:torsor}.
\end{proof}

We will denote $M$ with its Galois-module structure coming from these bijections by $M_{\phi}$, $M_{T}$, or $M_{L_0}$. Note that $T$, $L_0$, and $L^-$ are mutual resolvents.

\begin{examp} \label{ex:cubic_Gal_mod_2}
  If $M = C_3$, then the Galois module structures $M = M_T$ are in natural bijection with $C_2$-torsors over $K$, that is, quadratic \'etale extensions $T/K$, as we saw in Example \ref{ex:cubic_Gal_mod}. In this instance, $L^- = T$ and $L_0 = K \cross T$.
\end{examp}

The zeroth cohomology group $H^0(K, M)$ has a ready parametrization:
\begin{prop}[\textbf{a description of $H^0$}]\label{prop:H0}
  Let $M$ be a $K$-Galois module. The elements of $H^0(K, M)$ are in bijection with the degree-$1$ field factors of the extension $L_0$ corresponding to $M$ in the bijection of Proposition \ref{prop:Gal_mod}\ref{it:Gal_mod_L0}.
\end{prop}
\begin{proof}
  Proposition \ref{prop:Gal_mod} establishes an isomorphism of $G_K$-sets between the coordinates of $L_0$ and the points of $M$. A degree-$1$ field factor corresponds to an orbit of $G_K$ on $\Coord(L_0)$ of size $1$, which corresponds exactly to a fixed point of $G_K$ on $M$.
\end{proof}

Deeper and more useful is a description of $H^1$. We need the notion of the \emph{holomorph} of a group (see \cite[\textsection 6.3]{HallGroups}; compare \cite{LemHol25}):
\begin{prop}
  For an abelian group $M$, denote by $\Hol M$ the group of affine-linear transformations of $M$; that is, maps
  \[
    \lambda_{a,t}(x) = a x + t, \quad a \in \Aut M, \quad t \in M
  \]
  composed of an automorphism and a translation, the group operation being composition. Then $\Hol M = M \rtimes \Aut M$ is the semidirect product under the natural action of $\Aut M$ on $M$. In particular, there is a split exact sequence
  \begin{equation}
    0 \to M \to \Hol M \to \Aut M \to 0.
  \end{equation}
\end{prop}
\begin{proof}
  Simply compute that
  \begin{equation} \label{eq:semidirect}
    \lambda_{a,t} \circ \lambda_{b,u} = \lambda_{a b, b t + u}
  \end{equation}
  and note that this is the same as the group law on the semidirect product.
\end{proof}

Since an affine-linear map is a bijection, we have an embedding
\[
  \Hol M \hookrightarrow \Sym M.
\]
\begin{thm}[\textbf{a description of $H^1$}]\label{thm:H1}
  Let $M = M_{\phi} = M_{L_0}$ be a Galois module.
  \begin{enumerate}[$($a$)$]
    \item \label{it:z1_hom} $Z^1(K,M)$ is in natural bijection with the set of continuous homomorphisms $\psi : G_K \to \Hol M$ such that the following triangle commutes:
    \begin{equation} \label{eq:tri_h1}
      \xymatrix{
        G_K \ar[r]^\psi \ar[dr]_\phi & \Hol M \ar[d]^\pi \\
        & \Aut M
      }
    \end{equation}
    \item \label{it:h1_hom} $H^1(K,M)$ is in natural bijection with the set of such $\psi : G_K \to \Hol M$ up to postconjugation by $M \subseteq \Hol M$.
    \item \label{it:h1_ext} $H^1(K,M)$ is also in natural bijection with the set of $(\Hol M)$-extensions $L/K$ (with respect to the embedding $\Hol M \hookrightarrow \Sym M$) equipped with an isomorphism from their resolvent $(\Aut M)$-torsor $T_L$ to $T$. The bijection proceeds as follows: if $L$ corresponds to a cocycle $\sigma : G \to M$, there is a labeling $\{\iota_x\}_x$ of the coordinates of $L$ with the elements of $M$ such that for all $g \in G$,
    \begin{equation} \label{eq:G_act_L}
      g \circ \iota_x = \iota_{g x + \sigma(g)}.
    \end{equation}
    \item \label{it:Res} If $K'/K$ is a field extension, the restriction map
    \[
      \Res : H^1(K, M) \to H^1(K', M)
    \]
    can be described as follows. If $\sigma \in H^1(K,M)$ corresponds to an \'etale algebra $L$ with isomorphism $\theta : T_L \to T$, then $\Res \sigma$ corresponds to the \'etale algebra $L' = L \tensor_K K'$ with isomorphism $\theta' = \theta \tensor_{K} L' : T_{L'} \isom T_L \tensor_K K' \to T \tensor_K K'$.
  \end{enumerate}
\end{thm}
\begin{proof}
  By the standard construction of group cohomology, $Z^1$ is the group of continuous crossed homomorphisms
  \[
  Z^1(K, M) = \{\sigma : G_K \to M \mid \sigma(g h) = \sigma(g) + \phi(g) \sigma(h)\}.
  \]
  Send each $\sigma$ to the map
  \begin{align*}
    \psi : G_K &\to \Hol M \\
    g &\mapsto \lambda_{\phi(g), \sigma(g)}.
  \end{align*}
  Since the multiplication law on $\Hol M$ is given by
  \[
    (a_1, t_1) \cdot (a_2, t_2) = (a_1 a_2, t_1 + a_1 t_2),
  \]
  we see that the conditions for $\psi$ to be a homomorphism are exactly those for $\sigma$ to be a crossed homomorphism, establishing \ref{it:z1_hom}. For \ref{it:h1_hom}, we observe that adding a coboundary $\partial(x)(g) = g(x) - x$ to a crossed homomorphism $\sigma$ is equivalent to postconjugating the associated map $\psi : G_K \to \Hol M$ by $x$. As to \ref{it:h1_ext}, a $(\Hol M)$-extension carries the same information as a map $\psi$ up to conjugation by \emph{the whole of $\Hol M$}. Specifying the isomorphism from the resolvent $(\Aut M)$-torsor to $T$ means that the map $\pi \circ \psi = \phi : G_K \to \Aut M$ is known exactly, not just up to conjugation. Hence $\psi$ is known up to conjugation by $M$, and so we get a bijection to $H^1(K,M)$ by \ref{it:h1_hom}. The relation \ref{eq:G_act_L} follows by tracing the embedding $\Hol M \hookrightarrow \Sym M$ through the parametrization given by Propositions \ref{prop:G_sets} and \ref{prop:G_sets_G_strucs}. Part \ref{it:Res} is a straightforward consequence of Proposition \ref{prop:respects_base_chg}.
\end{proof}
\begin{rem}
  The zero cohomology class corresponds to the extension $L_0$ of Proposition \ref{prop:Gal_mod}\ref{it:Gal_mod_L0}, with its structure given by the embedding $\Aut M \hookrightarrow \Hol M$. This is the unique cohomology class whose corresponding $(\Hol M)$-extension has a field factor of degree $1$.
\end{rem}

If $M = M_{T}$ is a Galois module and $\sigma \in Z^1(K,M)$ is the Galois module corresponding to a $(\Hol M)$-extension $L/K$, we can also take the $ (\Hol M) $-closure of $L$ (as in Example \ref{ex:Cayley}), a $ (\Hol M) $-torsor $E$ which fits into the following diagram:
\begin{equation}
\begin{gathered}
\xymatrix@dr{
    E \ar@{-}[r]\ar@{-}[d] & T \ar@{-}[d] \\
    L \ar@{-}[r] & K }
\end{gathered}
\end{equation}
Because of the semidirect product structure of $ \Hol M $, we have $ E \isom L \tensor_K T $. To summarize, here are the permutation representations of finite groups that yield each of the \'etale algebras discussed here:
\begin{equation}
\xymatrix{
  0 \ar[r] & M \ar[r] & \Hol M \ar[r] \ar@{^{(}->}[ld]_{\text{yields $L$}} \ar@{^{(}->}[d]^{\text{yields $E$}} & \Aut M \ar[r] \ar@{^{(}->}[d]^{\text{yields $T$}} & 0 \\
  & \Sym M & \Sym(\Hol M) & \Sym(\Aut M)
}
\end{equation}

Understanding the group operation on $H^1(K,M)$ is a bit awkward in this framework, but at the very least, we can state the following:
\begin{prop}
  If coclasses $\sigma_1, \sigma_2, \sigma_1 + \sigma_2 \in H^1(K,M)$ correspond to \'etale algebras $L_1, L_2, L$ under Theorem \ref{thm:H1}\ref{it:h1_ext}, then $L$ is a subalgebra of $L_1 \tensor_K L_2$.
\end{prop}
\begin{proof}
  By Theorem \ref{thm:H1}\ref{it:h1_ext}, the coordinates $\iota^{(1)}_x, \iota^{(2)}_x$ of $L_1$ and $L_2$ can be labeled with the elements of $M$ such that for all $g \in G_K$,
  \begin{equation*}
    g \circ \iota^{(j)}_x = \iota^{(j)}_{g x + \sigma_j(g)}.
  \end{equation*}
  Now $L_1 \tensor L_2$ has $|M|^2$ coordinates, which can be labeled $\{\iota_{x_1,x_2}\}_{(x_1,x_2) \in M\cross M}$ such that
  \[
    \iota_{x_1,x_2}(\alpha_1 \tensor \alpha_2) = \iota^{(1)}_x(\alpha_1) \cdot \iota^{(2)}_x(\alpha_2).
  \]
  We compute that
  \[
    g \circ \iota_{x_1,x_2} = \iota_{g x_1 + \sigma_1(g), g x_2 + \sigma_2(g)}.
  \]
  Consider the following equivalence relation on $M \cross M$:
  \[
    (x_1, x_2) \sim (y_1, y_2) \iff x_1 + x_2 = y_1 + y_2.
  \]
  We compute that $\sim$ is $G_K$-invariant because, if $(x_1, x_2) \sim (y_1, y_2)$ and $g \in G$, then
  \begin{align*}
    \big(g x_1 + \sigma_1(g)\big) + \big(g x_2 + \sigma_2(g)\big)
    &= g(x_1 + x_2) + \sigma_1(g) + \sigma_2(g) \\
    &= g(y_1 + y_2) + \sigma_1(g) + \sigma_2(g) \\
    &= \big(g y_1 + \sigma_1(g)\big) + \big(g y_2 + \sigma_2(g)\big).
  \end{align*}
  Therefore, by Proposition \ref{prop:sub}, $\sim$ corresponds to a subextension $L \subseteq L_1 \tensor L_2$ whose coordinates are labeled by the equivalence classes of $\sim$. Identifying the equivalence class $\{(x_1, x_2) : x_1 + x_2 = x\}$ with the corresponding element $x \in M$, we have that $L$ has $n$ coordinates $\{\iota_x\}_{x \in M}$, which are permuted by
  \[
    g \circ \iota_x = \iota_{g x + \sigma_1(x) + \sigma_2(x)}.
  \]
  Hence $L$ corresponds to the coclass $\sigma_1 + \sigma_2$, as claimed.
\end{proof}

\subsection{Unramified cohomology}

If $K$ is a discretely valued field (e.g.\ $\FF_q\laurent{t}$ or $\QQ_p$), a coclass $\sigma \in H^1(K,M)$ is called \emph{unramified} if it is represented by a cocycle $\sigma : \Gal(K^\sep/K) \to M$ that factors through the Galois group $\Gal(K^{\ur}/K)$ of its maximal unramified extension. The subgroup of unramified coclasses is denoted by $H^1_\ur(K, M)$. If $M$ itself is unramified (that is, $M = M_T$ with $T$ unramified), this is equivalent to the associated \'etale algebra $L$ (equivalently $E$) being unramified.

\subsection{The Tate dual}

If $m$ is a positive integer not divisible by $\ch K$, one particularly simple $K$-Galois module is the group $\mu_m$ of $m$th roots of unity. Under Proposition \ref{prop:Gal_mod} it corresponds to the $(\ZZ/m\ZZ)^\cross$-torsor $T = L^- = K[X]/\Phi_m(X)$, where $\Phi_m$ is the cyclotomic polynomial, a direct summand of $L = K[X]/(X^m - 1)$.

If $M$ is a Galois module with exponent $m$, then
\[
  M' = \Hom(M, \mu_m)
\]
is also a Galois module, called the \emph{Tate dual} of $M$ (a Tate twist of the usual Pontryagin dual $M^* = \Hom(M, \QQ/\ZZ)$. The modules $M$ and $M'$ have the same order and are isomorphic as abstract groups, though not canonically; as Galois modules, they are frequently not isomorphic at all. However, there are canonical isomorphisms $\Aut M' \isom \Aut M$ and $M'' \isom M$.

\begin{examp}\label{ex:cubic_Tate_dual}
Assume that $\ch K \neq 2, 3$, and let $M$ be a Galois module of order $3$. Then $M = M_T$ corresponds to a quadratic algebra $T = K[\sqrt{D}]$ as noted in Example \ref{ex:cubic_Gal_mod}, then the relevant $\mu_m$ is
\[
  \mu_3 \isom M_{K[\sqrt{-3}]}.
\]
The Galois action on $M' = \Hom(M, \mu_3)$ depends on the Galois action on both $M$ and $\mu_3$. As can be seen from the results on $G_K$-sets of size $2$ presented in Knus and Tignol \cite{QuarticExercises}), we get
  \[
  M' = M_{K[\sqrt{-3D}]}.
  \]
  This underlies close connections between the quadratic extensions $K[\sqrt{D}]$ and $K[\sqrt{-3D}]$, such as the Scholz reflection principle \cite{ScholzRefl,EV}.
\end{examp}

\begin{examp}
  A module $M$ of underlying group $C_2 \cross C_2$ is always self-dual, regardless of what Galois-module structure is placed on it. This can be proved by noting that $\mu_2$ has the trivial action and $M$ has a unique alternating bilinear form (see Theorem \ref{thm:Tate_quartic} below).
\end{examp}

\begin{examp}
For $p \nmid \ch K$ a prime and $k$ an integer, we can consider the tensor power $M_k = \mu_p^{\tensor k}$ as a Galois module over $K$. This is a cyclic module of order $p$ on which a Galois action is defined as follows: if $g \in G_K$, then let $a \in \FF_p$ be the exponent for which $g(\zeta_p) = \zeta_p^a$, and let $g$ act on $M_k$ by multiplication by $a^k$. By Fermat's little theorem, the structure of $M_k$ depends only on the congruence class of $k$ modulo $p - 1$. We compute $M_k' = M_{1-k}$.
\end{examp}

In a few cases, $\Hol M$ is the full symmetric group $\Sym M$; in other words, every bijective self-map of $M$ is affine-linear. There are only four such cases:
\begin{itemize}
  \item $M = \{1\}$, $\Hol M \cong S_1$
  \item $M = C_2$, $\Hol M \cong C_2 \cong S_2$
  \item $M = C_3$, $\Hol M \cong C_3 \rtimes C_2 \cong S_3$
  \item $M = C_2 \cross C_2$, $\Hol M \cong (C_2 \cross C_2) \rtimes S_3 \cong S_4$.
\end{itemize}
In these cases, every \'etale algebra $L/K$ of degree $\size{M}$ has a unique $(\Hol M)$-structure and thus corresponds to a Galois cohomology element $\sigma \in H^1(K,M)$ for some Galois module structure on $M$. Excluding the degenerate first case, we consider the others:
\begin{examp}\label{ex:quad}
For $M = C_2$, the quadratic \'etale algebras over $K$ are parametrized by the group $H^1(K, C_2) \allowbreak = \Hom(G_K, C_2)$.
\end{examp}
\begin{examp}\label{ex:cubic}
For $M = C_3$, a cubic \'etale algebra $L$ has a quadratic resolvent algebra $T$, as noted in Example \ref{ex:rsv_n2}, and $L$ corresponds to a coclass $\sigma$ in the $3$-torsion group $H^1(K, M_T)$. The correspondence is not quite a bijection; instead, any field $L$ corresponds to a pair of inverse coclasses $\{\sigma, -\sigma\}$ $(\sigma \neq 0)$ in $H^1(K, M_T)$, once for each of the two identifications of the resolvent $T$ with itself. The zero coclass corresponds to the reducible algebra $L = K \cross T$.
\end{examp}
\begin{examp}\label{ex:quartic}
For $M = C_2 \cross C_2$, a quartic \'etale algebra $L$ has a cubic resolvent algebra $R$, as noted in Example \ref{ex:rsv43}. Taking $L^- = R$ yields a Galois-module structure $M_R$ on $M$, where $G_K$ permutes the three non-identity elements of $M$ in the same manner that it permutes the three coordinates of $R$. Then $L$ corresponds to at least one coclass $\sigma$ in the $2$-torsion group $H^1(K, M_R)$. The precise number of coclasses corresponding to a quartic algebra $L$ depends on the Galois group of $L$. For instance, if $G(L/K) \isom S_4$, then $\sigma$ appears only once since $R$ has no automorphisms. However, if $G(L/K) \isom V_4$, then $R \isom K \cross K \cross K$ is totally split, and $L$ corresponds to six elements in $H^1(K, M)$ which are transitively permuted by the action of $\Aut M \isom S_3$ on $H^1(K,M)$.
\end{examp}

\section{Understanding Galois cohomology via Kummer theory}
\label{sec:Kummer}

As a viewpoint for understanding Galois cohomology, \'etale algebras are already more manageable than maps out of an enormous Galois group. But in some cases the Galois cohomology group can be described even more explicitly. One such instance is \emph{Kummer theory:} If $\ch K \nmid n$, the Kummer exact sequence
\[
  0 \to \mu_n \to \(K^\sep\)^\cross \to \(K^\sep\)^\cross \to 0
\]
yields a parametrization
\[
  \Kum : K^\cross/\(K^\cross\)^n \isom H^1(K,\mu_n).
\]
Explicitly, it sends $a \in K^\cross$ to the coclass parametrized by
\[
  L = K[\!\sqrt[n]{a}] = K[X]/(X^n - a),
\]
a degree-$n$ algebra over $K$ with a $\Hol(\ZZ/n\ZZ)$-structure, since every $g \in G_K$ permutes the $n$ coordinates
\[
  \iota_{k} : X \mapsto \zeta_n^{k} \sqrt[n]{a}
\]
in an affine-linear fashion
\[
  g \circ \iota_t = \iota_{s k + t},
\]
where $s \in (\ZZ/n\ZZ)^\cross$ and $t \in \ZZ/n\ZZ$ are determined by $g(\zeta_n) = \zeta_n^s$ and $g(\!\sqrt[n]{a}) = \zeta_n^t\sqrt[n]{a}$.

\subsection{\texorpdfstring{$H^1(K,C_3)$}{H¹(K, C₃)} and cubic extensions}

Another, lesser-known case where a Kummer-theoretic parametrization describes the cohomology is that where $M$ has underlying group $C_3$.

\begin{prop}\label{prop:Kummer_cubic} Let $K$ be a field with $\ch K \neq 2, 3$, and let $ M $ be a $K$-Galois module with underlying group $ C_3 $. Let $T'$ be the quadratic \'etale algebra corresponding to the Tate dual $M'$ under the bijection of Proposition \ref{prop:Gal_mod}\ref{it:Gal_mod_L-}. Then we have a group isomorphism
  \[
  H^1(K, M) \isom T'^{N=1}/(T'^{N=1})^3
  \]
  in which $\delta \in T'^{N=1}/(T'^{N=1})^3$ corresponds to the cubic extension $L/K$ that is the image of the $K$-linear map
  \begin{align*}
    \kappa : K \cross T' &\to \(K^\sep\)^3 \\
    (a,\xi) &\mapsto \(a + \tr_{\(K^\sep\)^2/K^\sep} \xi \omega \sqrt[3]{\delta}\)_{\omega \in (\(K^\sep\)^2)^{N=1}[3]},
  \end{align*}
  where $\sqrt[3]{\delta} \in \(K^\sep\)^2$ is chosen to have norm $1$, and $\omega$ ranges through the set
  \[
  (\(K^\sep\)^2)^{N=1}[3] = \{(1;1), (\zeta_3; \zeta_3^2); ( \zeta_3^2; \zeta_3)\}
  \]
  of cube roots of $1$ in $\(K^\sep\)^2$ of norm $1$.
\end{prop}

\begin{rem}
  The construction $\kappa$ is none other than the classical solution of the general cubic equation by radicals. The passage from $T = K[\sqrt{D}]$ to its Tate dual $T' = K[\sqrt{-3D}]$ reflects the fact that solving a cubic with three real roots involves an intermediate square root of a negative quantity, a circumstance that vexed Renaissance algebraists and spurred the development of complex numbers in the 16th century \cite{ReadingBombelli}.
\end{rem}

\begin{proof}
The case $T' \isom K \cross K$ is a corollary of the preceding, since $T'^{N=1} \isom K^\cross$. So we assume that $T' \not\isom K \cross K$ is a field. Note that the restriction map $\Res : H^1(K,M) \to H^1(T',M)$ is an injection, having left inverse the negative of the corresponding corestriction map (as $\Cor \circ \Res$ acts by multiplication by $[T' : K] = 2$, which is negation since $H^1(K,M)$ is a $3$-torsion group). Now if $T = K[\sqrt{D}]$, then $T' = K[\sqrt{-3D}]$. Note that $M \isom \mu_3$ as $G_{T'}$-modules, so $H^1(T',M) \isom T'^\cross/(T'^\cross)^3$. Now we must determine which $\delta \in T'^\cross/(T'^\cross)^3$ correspond to coclasses restricted from $K$. Such coclasses must be invariant under the conjugation involution
\begin{align*}
  T'^\cross/(T'^\cross)^3 &\to T'^\cross/(T'^\cross)^3 \\
  \delta &\mapsto \bar\delta^{-1},
\end{align*}
where the exponent $-1$ appears because the isomorphism $M \isom \mu_3$ of $T'$-Galois modules is anti-invariant under conjugation.
So $N(\delta) = \delta \bar \delta = \alpha^3$ with $\alpha \in T'^\cross$, and the scaling $\delta \mapsto \delta(\alpha/\delta)^3$ normalizes the norm $N(\delta)$ to $1$. So the image of $\Res$ is contained in $T'^{N=1}/(T'^{N=1})^3$.

Conversely, given $\delta \in T'$ of norm $1$, it is easy to check that the map $\kappa$ as defined in the theorem statement is injective and that
\[
\kappa(a, \xi) \cdot \kappa(b, \eta) = \kappa\(ab + \tr_{T'/K}(\xi \bar\eta), a\eta + b\xi + \frac{\bar \xi \bar \eta}{\delta}\),
\]
implying that the image $L$ is a cubic algebra, necessarily Minkowski embedded since the three coordinates of $\kappa$ are linearly independent. Multiplying $\delta$ by a cube of norm $1$ does not change $L$. One can compute that $\kappa(a, \xi)$ has discriminant $-27(\xi^3 \delta - \bar \xi^3 \bar \delta)^2$, so $L$ has quadratic resolvent $K[\sqrt{-27 \cdot -3D}] = K[\sqrt{D}] = T$. To see the corresponding Kummer element, we pass to the base change
\begin{align*}
  L \tensor_{K} T' &= \left\{\(\alpha + \tr_{\(K^\sep\)^2/K^\sep} \xi \omega \sqrt[3]{\delta}\)_{\omega \in (\(K^\sep\)^2)^{N=1}[3]} : \alpha \in T', \xi \in T' \tensor_K T' \isom T' \cross T' \right\} \\
  &= \left\{\(\alpha + \beta \omega \sqrt[3]{\delta} + \gamma \omega^2 (\!\sqrt[3]{\delta})^{-1}\)_{\omega \in \mu_3} : \alpha, \beta, \gamma \in T'\right\} \\ &= T'[\!\sqrt[3]{\delta}].
\end{align*}
So all $\delta \in T'^{N=1}/(T'^{N=1})^3$ arise from $H^1(K,M)$.
\end{proof}

\begin{rem}
The theorem statement holds without change in characteristic $2$, but one must in the proof replace $T = K[\sqrt{D}]$ with the appropriate Artin-Schreier extension. In characteristic $3$, the corresponding statement is
\[
  H^1(K, M) \isom T^{\tr = 0} / \wp(T^{\tr = 0})
\]
where $\wp(\xi) = \xi^3 - \xi$ is the Artin-Schreier map. We omit the proofs.
\end{rem}

\subsection{\texorpdfstring{$H^1(K,C_2 \cross C_2)$}{H¹(K, C₂×C₂)} and quartic extensions}

A similar method yields the following:
\begin{prop}[\cite{QuarticExercises}]\label{prop:Kummer_quartic} Let $K$ be a field with $\ch K \neq 2$, and let $ M $ be a $K$-Galois module with underlying group $ C_2 \cross C_2 $. Let $R$ be the cubic \'etale algebra corresponding to $M$ under the bijection of Proposition \ref{prop:Gal_mod}\ref{it:Gal_mod_L-}. Then we have a group isomorphism
  \[
  H^1(K, M) \isom R^{N=1}/(R^{N=1})^2
  \]
  in which $\delta \in R^{N=1}/(R^{N=1})^2$ corresponds to the quartic extension $L/K$ generated by the image of the $K$-linear map
  \begin{align*}
    \kappa : K &\to \(K^\sep\)^4 \\
    \xi &\mapsto \(\tr_{\(K^\sep\)^3/K} \xi \omega \sqrt{\delta}\)_\omega \in (\(K^\sep\)^2)^{N=1}[3],
  \end{align*}
  where $\sqrt{\delta} \in \(K^\sep\)^2$ is chosen to have norm $1$, and $\omega$ ranges through the set
  \[
  (\(K^\sep\)^3)^{N=1}[2] = \{(1;1;1), (1;-1;-1), (-1;-1;1), (-1;1;-1)\}
  \]
  of cube roots of $1$ in $\(K^\sep\)^3$ of norm $1$. Indeed
  \[
  L = K + \kappa(R).
  \]
\end{prop}
\begin{proof}
  See Knus and Tignol \cite[Proposition~5.13]{QuarticExercises}.
\end{proof}

\subsection{\texorpdfstring{$H^1(K,C_4)$}{H¹(K, C₄)} and quartic extensions}
Another small abelian group is $M = C_4$. Its automorphism group is of order $2$, so, just as in Example \ref{ex:cubic_Gal_mod}, Galois module structures are parametrized by quadratic extensions. Notably, $\Hol C_4 \isom D_4 \isom C_2 \wr C_2$, so by Example \ref{ex:D4}, a $D_4$-extension is none other than a quadratic extension of a quadratic extension. We can describe the first cohomology of cyclic Galois modules of order $4$ in the following way.
\begin{thm} \label{thm:D4}
  Let $K$ be a field, $\ch K \neq 2$, and let $M$ be the cyclic Galois module over $4$ corresponding to a quadratic extension $K[\sqrt{D}]$. Then we have a group isomorphism
  \begin{equation}
    H^1(K,M) \isom \{(\alpha, c) \in K[\sqrt{-D}]^\cross \cross K^\cross : N(\alpha) = c^4\}/\{(\beta^4, N(\beta)) : \beta \in K[\sqrt{-D}]^\cross\} \label{eq:D4_main}
  \end{equation}
  in which the coclass corresponding to a pair $(a + b\sqrt{-D}, c)$ ($a,b,c \in K$) has corresponding quartic algebra
  \[
    L = K\left[\sqrt{2a + 2c^2}, \sqrt{2c + \sqrt{2a + 2c^2}}\right].
  \]
\end{thm}
\begin{proof}
We have $T = K[\sqrt{D}]$, $T' = K[\sqrt{-D}]$. If $T' = K \cross K$, the conclusions follow easily from Kummer theory, so we assume that $T'$ is a field. As Galois modules over $T'$, we have an isomorphism $\xi : M \to \mu_4$, so by Kummer theory, $H^1(T', M) \isom T'^\cross / (T'^\cross)^4$. Applying this to the restriction of some cocycle $\sigma \in Z^1(K,M)$, we obtain that there exists $\alpha \in K^\cross$ such that, for all $g \in G_{T'}$,
\begin{equation} \label{eq:x_D4_Kum}
  \xi\big(\sigma(g)\big) = \frac{g(\!\sqrt[4]{\alpha})}{\sqrt[4]{\alpha}}.
\end{equation}
To make sense of this formula, it is necessary to fix a particular root $\sqrt[4]{\alpha} \in K^\sep$; the scaling $\sqrt[4]{\alpha} \mapsto i \sqrt[4]{\alpha}$ corresponds to offsetting $\sigma$ by the coboundary $g \mapsto g(x) - x$ where $x = \xi^{-1}(i) \in M$.

Now let $r \in G_K$ such that $r(\sqrt{-D}) = -\sqrt{-D}$, so $G_K = G_{T'} \sqcup r G_{T'}$. Note that $\xi (g x) = g \bigl(\xi x\bigr)$ for all $g \in G_{T'}$ and $x \in M$ (this is what it means for $\xi$ to be a morphism of $T'$-Galois modules), but
\begin{equation}
  \xi (r x) = r \bigl(\xi x\bigr)^{-1}.
\end{equation}
Set
\[
  c = \sqrt[4]{\alpha} \cdot r\bigl(\!\sqrt[4]{\alpha}\bigr) \cdot \xi\bigl(\sigma(r)\bigr).
\]
Observe that $c$ is independent of the choice of $r$, since if $r' = r \cdot g$ with $g \in G_{T[i]}$, then
\begin{align*}
  c' &= \sqrt[4]{\alpha} \cdot r'\bigl(\!\sqrt[4]{\alpha}\bigr) \cdot \xi\bigl(\sigma(r')\bigr) \\
  &= \sqrt[4]{\alpha} \cdot r\bigl(g(\!\sqrt[4]{\alpha})\bigr) \cdot \xi\bigl(\sigma(r) + r \sigma(g)\bigr) \\
  &= \sqrt[4]{\alpha} \cdot r\(\!\sqrt[4]{\alpha} \cdot \xi(\sigma(g)) \) \cdot \xi\(\sigma(r) + r \sigma(g)\) \\
  &= \sqrt[4]{\alpha} \cdot r\bigl(\!\sqrt[4]{\alpha}\bigr) \cdot \xi\bigl(\sigma(r)\bigr) \cdot \xi\bigl(\sigma(g) - \sigma(g)\bigr) \\
  &= c.
\end{align*}
We claim that $c \in K^\cross$. For this it is enough to show that $c$ is Galois invariant. If $g \in G_{T'}$, then
\begin{align*}
  g(c) &= g\bigl(\!\sqrt[4]{\alpha}\bigr) \cdot gr\bigl(\!\sqrt[4]{\alpha}\bigr) \cdot g\Bigl(\xi\bigl(\sigma(r)\bigr)\Bigr) \\
  &= \sqrt[4]{\alpha} \cdot \xi\bigl(\sigma(g)\bigr) \cdot gr\bigl(\!\sqrt[4]{\alpha}\bigr) \cdot \xi\bigl(g \sigma(r)\bigr) \\
  &= \sqrt[4]{\alpha} \cdot gr\bigl(\!\sqrt[4]{\alpha}\bigr) \cdot \xi\bigl(\sigma(g) + g\sigma(r)\bigr) \\
  &= \sqrt[4]{\alpha} \cdot gr\bigl(\!\sqrt[4]{\alpha}\bigr) \cdot \xi\bigl(\sigma(gr)\bigr) \\
  &= c,
\end{align*}
since changing $r$ to $gr \in r G_{T'}$ does not change $c$. As to invariance by $r G_{T'}$, we similarly have
\begin{align*}
  r^{-1}(c) &= r^{-1}\bigl(\!\sqrt[4]{\alpha}\bigr) \cdot \sqrt[4]{\alpha} \cdot r^{-1}\Bigl(\xi\bigl(\sigma(r)\bigr)\Bigr) \\
  &= \sqrt[4]{\alpha} \cdot r^{-1}\bigl(\!\sqrt[4]{\alpha}\bigr) \cdot \xi\big(-r^{-1}\sigma(r)\big) \\
  &= \sqrt[4]{\alpha} \cdot r^{-1}\bigl(\!\sqrt[4]{\alpha}\bigr) \cdot \xi\big(\sigma(r^{-1})\big) \\
  &= c.
\end{align*}
So $c \in K^\cross$. If $\alpha$ is scaled by $\beta^4$, then $c$ scales by $N(\beta)$. So we get a map
\[
  \Kum : H^1(K,M) \to \{(\alpha, c) \in K[\sqrt{-D}]^\cross \cross K^\cross : N(\alpha) = c^4\}/\{(\beta^4, N(\beta)) : \beta \in K[\sqrt{-D}]^\cross\}
\]
which one easily checks to be a group homomorphism. For the inverse, given the Kummer datum $(\alpha, c)$, define $\sigma$ by
\begin{equation} \label{eq:x_cocy_inv}
  \sigma(g) = \begin{cases}
    \xi^{-1}\left(\dfrac{g\bigl(\!\sqrt[4]{\alpha}\bigr)}{\sqrt[4]{\alpha}}\right), & g \in G_{T'} \\
    \xi^{-1}\left(\dfrac{c}{\sqrt[4]{\alpha} \, g\bigl(\!\sqrt[4]{\alpha}\bigr)}\right), & g \in G_K \setminus G_{T'}.
  \end{cases}
\end{equation}

It is a straightforward matter to check that $\sigma$ is a cocycle, that the choice of $4$th root can only affect $\sigma$ by a coboundary, and that this construction is inverse to $\Kum$ as defined above.

For the last claim, let a Kummer datum $(\alpha, c) = (a + b\sqrt{D}, c)$ correspond to a cocycle $\sigma \in Z^1(K,M)$. From \eqref{eq:x_cocy_inv}, we see that the $g$ with $\sigma(g) = 0$ fix the element
\[
  \theta_0 = \sqrt[4]{\alpha} + \frac{c}{\sqrt[4]{\alpha}}.
\]
More generally, if for $x \in M$ we set
\[
  \theta_a = \xi(x) \sqrt[4]{\alpha} + \frac{c}{\xi(x) \sqrt[4]{\alpha}},
\]
 we have that 
\begin{equation} \label{eq:x_theta_permute}
  g(\theta_a) = \theta_{\chi_D(g) x + \sigma(x)}
\end{equation}
where $\chi_D : G_K \to \mu_2$ is the quadratic character $\chi_D(g) = g(\sqrt{D})/\sqrt{D}$. Conversely, assuming the $\theta_a$ are distinct (which they are, except in the case that $\alpha = \pm c^2 \in K$, which can be avoided by rescaling $\alpha$ by a suitably generic $4$th power in $K[\sqrt{-D}]$), the relation \eqref{eq:x_theta_permute} determines $\sigma$. A brief calculation shows that the $\theta_a$ are the four roots of the quartic
\[
  (\theta^2 - 2c)^2 = 2a + 2c^2.
\]
Hence we recover $\sigma$ via the parametrization of Theorem \ref{thm:H1}, for we now see that $\psi_\sigma : G_K \to \Hol M \isom D_4$ is the map that records how $G_K$ permutes the four vertices of the square
\begin{equation} \label{eq:square1}
  \setlength{\unitlength}{1in}
  \begin{picture}(1.4,1.4)(-0.2,-0.2)%
    \put(0,0){\framebox(1,1){}}
    \put(0,0){\line(1,1){1}}
    \put(0,1){\line(1,-1){1}}
    \put(0,1){\makebox(0,0)[br]{$\theta_0 = \sqrt{2c + \sqrt{2a + 2c^2}}$}}
    \put(1,1){\makebox(0,0)[bl]{$\theta_1 = \sqrt{2c - \sqrt{2a + 2c^2}}$}}
    \put(1,0){\makebox(0,0)[tl]{$\theta_2 = -\sqrt{2c + \sqrt{2a + 2c^2}},$}}
    \put(0,0){\makebox(0,0)[tr]{$\theta_3 = -\sqrt{2c - \sqrt{2a + 2c^2}}$}}
  \end{picture}
\end{equation}
yielding the claimed correspondence.
\end{proof}
\begin{rem}
  There are several natural involutions on $H^1(K,M)$, which we trace through the correspondences of Theorems \ref{thm:H1} and \ref{thm:D4} to allay any confusion:
  \begin{enumerate}
    \item Negation in $H^1(K,M)$ corresponds to postcomposing $\psi_\sigma : G_K \to D_4$ with either of the two inner automorphisms of $D_4$ that are conjugation by a reflection. (There are four reflections in $D_4$, but due to the presence of a central element, they only define two automorphisms.) This sends $(\alpha, c)$ to $(\bar \alpha, c)$, which is equivalent to the expected inverse $(1/\alpha, 1/c)$ under \eqref{eq:D4_main} (the ratio of the Kummer data being $(N(\alpha), c^2) = (c^4, N(c))$).
    \item Passing to the mirror field (compare Example \ref{ex:D4}) corresponds to postcomposing $\psi_\sigma : G_K \to D_4$ with an outer automorphism of $D_4$ that rotates the square \ref{eq:square1} by $45^\circ$. In other words, we ask how $G_K$ permutes the four numbers $\theta_k + \theta_{k+1}$:
    \begin{equation} \label{eq:square2}
      \setlength{\unitlength}{1in}
      \begin{picture}(1.8,1.8)(-0.4,-0.4)%
        \put(0,0){\framebox(1,1){}}
        \put(0,0){\line(1,1){1}}
        \put(0,1){\line(1,-1){1}}
        \put(0,1){\makebox(0,0)[br]{$\theta_0$}}
        \put(1,1){\makebox(0,0)[bl]{$\theta_1$}}
        \put(1,0){\makebox(0,0)[tl]{$\theta_2$}}
        \put(0,0){\makebox(0,0)[tr]{$\theta_3$}}
        \thicklines
        \put(0.5,-0.21){\line(0,1){1.42}}
        \put(-0.21,0.5){\line(1,0){1.42}}
        \put(-0.21,0.5){\line(1,1){0.71}}
        \put(-0.21,0.5){\line(1,-1){0.71}}
        \put(-0.21,0.5){\line(1,1){0.71}}
        \put(-0.21,0.5){\line(1,-1){0.71}}
        \put(1.21,0.5){\line(-1,1){0.71}}
        \put(1.21,0.5){\line(-1,-1){0.71}}
        \put(0.5,1.21) {\makebox(0,0)[b]{$\theta_0 + \theta_1 = \sqrt{4 c + 2\sqrt{-2 a + 2 c^2}}$}}
        \put(1.26,0.5) {\makebox(0,0)[l]{$\theta_1 + \theta_2 = \sqrt{4 c - 2\sqrt{-2 a + 2 c^2}}$}}
        \put(0.5,-0.21){\makebox(0,0)[t]{$\theta_2 + \theta_3 = -\sqrt{4 c + 2\sqrt{-2 a + 2 c^2}}$\rlap{.}}}
        \put(-0.26,0.5){\makebox(0,0)[r]{$\theta_3 + \theta_0 = -\sqrt{4 c - 2\sqrt{-2 a + 2 c^2}}$}}
      \end{picture}
    \end{equation}
    Here it is understood that there is a choice of signs for the square roots that is consistent with the way we labeled the $\theta_i$ earlier.
    From this we see that the mirror field $L'$ is generated by
    \[
      \sqrt{4c + 2 \sqrt{-2a + 2c^2}} = \sqrt{2c' + \sqrt{2a' + 2c'^2}}
    \]
    where $c' = 2c$, $a' = -4a$, $b' = -4b$. Accordingly, the mirror field relation is not an automorphism of $H^1(K,M)$, but rather a translation by the special element $\sigma_{\mathrm{mir}} \in H^1(K,M)$ whose Kummer datum is $(-4,2)$, and whose corresponding $D_4$-quartic algebra is the split algebra $K[\sqrt{D}] \cross K[\sqrt{D}]$.
    
    There is a sign ambiguity on $b'$ that corresponds to negating the coclass of the mirror field. That we have made a consistent choice of sign appears more clearly in the following perspective: if
    \[
      \theta_k = i^k \sqrt[4]{\alpha} + \frac{c}{i^k \sqrt[4]{\alpha}},
    \]
    then
    \begin{align*}
      \theta_k + \theta_{k+1} &= i^k (1 + i) \sqrt[4]{\alpha} + \frac{c(1 - i)}{i^k \sqrt[4]{\alpha}} \\
      &= i^k \sqrt[4]{-4\alpha} + \frac{2c}{i^k \sqrt[4]{-4\alpha}}
    \end{align*}
    generates a field with Kummer datum $(-4\alpha, 2c)$.
  \end{enumerate}
\end{rem}

\begin{qn}
  Is there a suitable analogue of the theorems in this section for every Galois module $M$? Specifically, if the Tate dual $M'$ corresponds to an $(\Aut M)$-torsor $T'$ over a field $K$ of characteristic not dividing $|M|$, we seek an integer $k \geq 0$ and subgroups $H \subseteq G \subseteq (T'^\cross)^k$, such that there is a natural identification
  \[
    H^1(K,M) \isom G/H.
  \]
\end{qn}

\section{The local Tate pairing}
\label{sec:tate}

Assume now that $ K $ is a \emph{local} field, that is, a finite extension of $\QQ_p$ or $\FF_p\laurent{t}$ for some prime $p$. Fundamental to class field theory is the \emph{Tate pairing}, a perfect pairing given by the cup product
\[
\langle \, , \, \rangle_T : H^1(K, M) \cross H^1(K, M') \to H^2(K, \mu_m) \isom \mu_m.
\]
In many cases this pairing can be described explicitly. For instance, if $M \cong C_m $ has \emph{trivial} $ G_K $-action, then $ M' \cong \mu_m $, and we have a Tate pairing
\[
\langle \, , \, \rangle_T : H^1(K, C_m) \cross H^1(K, \mu_m) \to \mu_m
\]
Now $H^1(K, C_m) \cong \Hom(K, C_m)$ parametrizes $C_m$-torsors $L$, while by Kummer theory, $H^1(K, \mu_m) \cong K^\cross / (K^\cross)^m$. The Tate pairing in this case is none other than the \emph{Artin symbol} $\phi_L : K^\cross \to \Gal(L/K) \to \mu_m$ whose kernel is the norm group $N_{L/K}(L^\cross)$ (see Neukirch \cite{NeukirchCoho}, Prop.~7.2.13). If, in addition, $\mu_m \subseteq K$, then $H^1(K, C_m)$ is also isomorphic to $K^\cross/\( K^\cross\) ^m$, and the Tate pairing is an alternating pairing
\[
\langle \, , \, \rangle : K^\cross/\( K^\cross\) ^m \cross K^\cross/\( K^\cross\) ^m \to \mu_m
\]
classically called the \emph{Hilbert symbol} (or \emph{Hilbert pairing}). It is defined in terms of the Artin symbol by
\begin{equation}
  \<a, b\> = \phi_{K[\!\sqrt[m]{b}]}(a).
\end{equation}
In particular, $\<a, b\> = 1$ if and only if $a$ is the norm of an element of $K[\!\sqrt[m]{b}]$. This can also be described in terms of the splitting of an appropriate Severi-Brauer variety; for instance, if $m = 2$, we have $ \<a,b\> = 1 $ exactly when the conic
\[
ax^2 + by^2 = z^2
\]
has a $K$-rational point. See also Serre (\cite{SerreLF}, \textsection\textsection XIV.1--2). (All identifications between pairings here are up to sign; the signs are not consistent in the literature.) To finish this paper, we find additional cases where the Tate pairing can be expressed in terms of the more explicit Hilbert pairing.

We extend the Hilbert pairing to \'etale algebras in the obvious way: if $L = K_1 \cross \cdots \cross K_s$ is a product of finite separable extensions, then
\[
\<(a_1;\ldots;a_s), (b_1;\ldots;b_s)\>_L := \<a_1, b_1\>_{K_1} \cdot \cdots \cdot \<a_s, b_s\>_{K_s}.
\]
Note that if $a$ is a norm from $L[\!\sqrt[m]{b}]$ to $L$, then $\<a, b\>_L = 1$, but the converse no longer holds. We then have the following:

\begin{prop}[\textbf{The Tate pairing for Galois modules of order 3}]\label{prop:Tate_pairing} Let $ K $ be a local field, $\ch K \neq 3$. Let $ M $ and $M'$ be a dual pair of Galois modules of order $3$, and let $T$, $T'$ be the corresponding $C_2$-torsors in the bijection of Proposition \ref{prop:Gal_mod}\ref{it:Gal_mod_T}. Then the Tate pairing
  \[
  \<\bullet,\bullet\> : H^1(K, M) \cross H^1(K, M') \to H^2(K, \mu_m) \isom C_m
  \]
  is the restriction of the Hilbert pairing on $ E := T[\mu_3] $, which naturally contains both $T$ and $T'$.
\end{prop}
\begin{proof}
  Note that $E$ is an extension that splits both $M$ and $M'$. Since $[E : K] = 4$ is coprime to $\size{M} = 3$, restriction and corestriction embed $H^1(K,M)$ and $H^1(K, M')$ as direct summands of $H^1(E, C_3)$. This allows us to reduce to the case where both $M$ and $M'$ are trivial, and then the conclusion follows from the definition of the Hilbert pairing.
\end{proof}

\begin{rem}
  A similar result holds for modules $M$ of underlying group $C_p$ for any prime $p$.
\end{rem}

\begin{thm}[\textbf{The Tate pairing for Galois modules of type $C_2 \cross C_2$}] \label{thm:Tate_quartic} Let $ K $ be a local field, $\ch K \neq 2$. Let $M$ be a Galois module whose underlying group is $C_2 \cross C_2$ and whose Galois-module structure corresponds to an \'etale algebra $R$ under the bijection of Proposition \ref{prop:Gal_mod}\ref{it:Gal_mod_L-}. Then under the parametrization of Proposition \ref{prop:Kummer_quartic}, the Tate pairing
  \[
  \<\bullet,\bullet\> : H^1(K, M) \cross H^1(K, M) \to H^2(K, \mu_2) \isom \mu_2
  \]
  is the restriction of the Hilbert pairing on $R$.
\end{thm}
\begin{proof}
  Note that $M$ has a unique, and therefore Galois-invariant, nonzero alternating bilinear form
  \begin{align*}
    \epsilon : M \cross M &\to \FF_2 \\
    (x,y) &\mapsto \sum_{\substack{\chi \in \Hom(M, \FF_2) \\ \chi \neq 0}} \chi(x) \chi(y) \\
    &= \begin{cases}
      0 & \text{if $x = 0$ or $y = 0$ or $x = y$} \\
      1 & \text{otherwise,}
    \end{cases}
  \end{align*}
  which allows us to identify $M$ with its Tate dual.
  
  Let $ R_1,\ldots, R_\ell $ be the field factors of $ R $; each $ R_i $ corresponds to an orbit $ G_K \psi_i $ on $ M $. Denote by $\chi_i$ the corresponding map $\chi_i(x) = \epsilon(x, \psi) : M \to \mu_2$. Let $ \sigma, \tau \in H^1(K, M) $, and let
  \[
    \Kum : H^1(K,M) \to R^{N=1}/(R^{N=1})^2
  \]
  be the Kummer map from Proposition \ref{prop:Kummer_quartic}. We have
  \begin{align*}
    \<\Kum(\sigma), \Kum(\tau)\>_{R} &= \prod_{R_i} \<\Kum(\sigma), \Kum(\tau)\>_{R_i} \\
    &= \prod_{R_i} \inv_{R_i} \big(\chi_{i*} \Res^K_{R_i}\sigma \cup \chi_{i*} \Res^{K}_{R_i}\tau\big),
  \end{align*}
  where $\inv_F : H^2(F, \mu_2) \to \mu_2$ is the invariant map from class field theory. Using the standard fact $ \inv_{R_i} = \inv_K \circ \Cor^{K}_{R_i}$, we have
  \begin{align*}
    \<\Kum(\sigma), \Kum(\tau)\>_{R} &= \inv_K \sum_{R_i} \Cor^{K}_{R_i} \big(\chi_{i*} \Res^K_{R_i}\sigma \cup \chi_{i*} \Res^{K}_{R_i}\tau\big) \\
    &= \inv_K \sum_{R_i} \Cor^{K}_{R_i} (\chi_i \tensor \chi_i )_* \Res^K_{R_i}(\sigma \cup \tau).
  \end{align*}
  where $ \chi_i \tensor \chi_i : M \tensor M \to \mu_2 $. We now apply the following lemma, which generalizes results seen in the literature (see \cite[\textsection I.2.5]{SerreGC}).
  \begin{lem}\label{lem:cor_res}
    Let $H \subseteq G$ be an open subgroup of finite index in a profinite group $G$. Let $X$ and $Y$ be $G$-modules, and let $f : X \to Y$ be a map that is $H$-linear (but not necessarily $G$-linear). Denote by $\tilde f$ the $G$-linear map
    \begin{align*}
      \tilde{f} : X &\to Y \\
      \tilde{f}(x) &= \sum_{gH \in G/H} g f g^{-1}(x).
    \end{align*}
    Let $\sigma \in H^n(G, X)$. Then
    \[
      \Cor_H^G(f_* \Res_H^G \sigma) = \tilde{f}_* \sigma.
    \]
  \end{lem}
  \begin{proof}
    First note that $\tilde f$ is well-defined: since $f$ is $H$-linear, we have $hfh^{-1} = f$ for all $h \in H$, so $gfg^{-1}$ depends only on the coset $gH \in G/H$.
    
    We will prove this lemma using dimension shifting (see \cite[\textsection I.3]{NeukirchCoho}), which we explain here.
    
    First, we can assume that $G$ and $H$ are finite, since the $G$-module structures on $X$ and $Y$ as well as any particular cocycle $x$ factor through some finite quotient of $G$. This avoids topological complications in the work that follows.
    
    We use induction on $n$. The base case $n = 0$ can be calculated explicitly: Here $\sigma = x \in X$ is fixed under $G$, and
    \begin{align*}
      \Cor_H^G(f_* \Res_H^G \sigma) &= \Cor_H^G\big( f(x)\big) \\
      &= \sum_{gH \in G/H} gf(x) \\
      &= \sum_{gH \in G/H} gfg^{-1}(x) \\
      &= \tilde f(x)
    \end{align*}
    where the penultimate step uses that $x$ is fixed under $G$.
    
    Now assume the lemma is known for $H^{n-1}$. Passing to the induced module, we get a commutative diagram of $H$-modules with exact rows
    \[
      \xymatrix{
        0 \ar[r] & X \ar[r]\ar[d]^f & \Ind_G X \ar[r]\ar[d]^{\Ind_G f} & X_1 \ar[r] \ar[d]^{f_1} & 0 \\
        0 \ar[r] & Y \ar[r] & \Ind_G Y \ar[r] & Y_1 \ar[r] & 0\rlap{,}
      }
    \]
    (explicitly, $\Ind_G X = X \tensor_\ZZ \ZZ[G]$ and $X_1 = X \tensor_\ZZ (\ZZ[G]/\ZZ)$). This yields a commutative diagram in cohomology, with exact rows:
    \begin{equation} \label{eq:coho_cd}
      \xymatrix{
        H^{n-1}(G, X_1) \ar[r] \ar[d]^{\Res^{G}_H} & H^n(G, X) \ar[r]\ar[d]^{\Res^{G}_H} & H^n(G, \Ind_G X)\ar[d]^{\Res^{G}_H} \\
        H^{n-1}(H, X_1) \ar[r] \ar[d]^{f_{1*}} & H^n(H, X) \ar[r]\ar[d]^{f_*} & H^n(H, \Ind_G X)\ar[d]^{\Ind_G f} \\
        H^{n-1}(H, Y_1) \ar[r] \ar[d]^{\Cor^{G}_H} & H^n(H, Y) \ar[r]\ar[d]^{\Cor^{G}_H} & H^n(H, \Ind_G Y)\ar[d]^{\Cor^{G}_H} \\
        H^{n-1}(G, Y_1) \ar[r] & H^n(G, Y) \ar[r] & H^n(G, \Ind_G Y)
      }
    \end{equation}
    Because an induced module is cohomologically trivial, the right-hand column of this diagram is all zero. Also, by the induction hypothesis, the composition of the left-hand column is $(\widetilde{f_1})_* = (\tilde{f})_{1*}$. It follows that the composition of the middle column is $\tilde{f}_*$, as desired.
  \end{proof}
  Applying this lemma with $ G = G_K $, $ H = G_{R_i} $, and $ f = \chi_i \tensor \chi_i : M \tensor M \to \mu_2 $, we get
  \begin{align*}
    \<\Kum(\sigma), \Kum(\tau)\>_{R} &= \inv_K \sum_{R_i} \Big(\sum_{g \in G_K/G_{R_i}} (\chi_i \circ g^{-1} \tensor \chi_i \circ g^{-1}) \Big)_* (\sigma \cup \tau), \\
    &= \inv_K \sum_{R_i} \Big(\sum_{g \in G_K/G_{R_i}} \big(\chi_{i,g} \tensor \chi_{i,g}\big)\Big)_* (\sigma \cup \tau), \\
  \end{align*}
  where $ \chi_{i,g}(x) = \chi(g^{-1}(x)) $ are given by the natural action. Now the outer sum runs over all $ G_K $-orbits of $ M $ while the inner sum runs over the elements of each orbit, so we get
  \begin{align*}
    \<\Kum(\sigma), \Kum(\tau)\>_{R} &= \inv_K \Big(\sum_{\substack{\chi : M \to C_2 \\ \chi \neq 0}} \chi \tensor \chi) \Big)_* (\sigma \cup \tau) \\
    &= \inv_K \epsilon_* (\sigma \cup \tau) \\
    &= \<\sigma, \tau\>_{\text{Tate}},
  \end{align*}
  as desired.
\end{proof}

\begin{qn}
  Can the Tate pairing for modules with underlying group $C_4$ be described by Hilbert symbols in the Kummer data $(\alpha, c)$ of the parametrization of Theorem \ref{thm:D4}?
\end{qn}

\bibliography{../Master}
\bibliographystyle{alpha}

\end{document}